\documentclass[english]{article}
\usepackage[T1]{fontenc}
\usepackage[latin9]{inputenc}
\usepackage{geometry}
\usepackage{bm}
\usepackage{amsmath}
\usepackage{amsthm}
\usepackage{amssymb}
 \usepackage{color} 
 \usepackage{comment}

\makeatletter

\providecommand{\tabularnewline}{\\}

\theoremstyle{plain}
\newtheorem{thm}{\protect\theoremname}
  \theoremstyle{definition}
  \newtheorem{defn}[thm]{\protect\definitionname}
    \theoremstyle{lemma}
  \newtheorem{lem}[thm]{\protect\lemmaname}
  \theoremstyle{remark}
  \newtheorem{rem}[thm]{\protect\remarkname}
  \theoremstyle{definition}
  \newtheorem{example}[thm]{\protect\examplename}

\makeatother

\usepackage{babel}
  \providecommand{\definitionname}{Definition}
  \providecommand{\examplename}{Example}
  \providecommand{\remarkname}{Remark}
\providecommand{\theoremname}{Theorem}
\providecommand{\lemmaname}{Lemma}

\begin{document}

\title{The Tammes Problem in $\mathbb{R}^{n}$ and Linear Programming Method}

\author{Yanlu Lian
\thanks{
School of Mathematics, Hangzhou Normal University, Hangzhou, 311121, China (E-mail: yllian@hznu.edu.cn). She is the corresponding author. },
Qun Mo
\thanks{Zhejiang University, Hangzhou, 310027, China  (E-mail: moqun@zju.edu.cn).},
and Yu Xia
\thanks{School of Mathematics, Hangzhou Normal University, Hangzhou, 311121, China (E-mail: yxia@hznu.edu.cn). \protect\\
This work is supported by the NSFC grant (12271133, U21A20426, 11901143,11871481,11971427), the key project of Zhejiang Provincial Natural Science Foundation under grant number LZ23A010002 and STI2030-Major Project 2021ZD020410, the Scientific Research Fund of Zhejiang Provincial Education Department (Y202250092) and the Natural Science Foundation of Zhejiang Provincial (LQ24A010008).  }\\}
\date{}
\maketitle

\begin{abstract}
The Tammes problem delves into the optimal arrangement of $N$ points on the surface of the $n$-dimensional unit sphere (denoted as $\mathbb{S}^{n-1}$), aiming to maximize the minimum distance between any two points. In this paper, we articulate the sufficient conditions requisite for attaining the optimal value of the Tammes problem for arbitrary $n, N \in \mathbb{N}^{+}$, employing the linear programming framework pioneered by Delsarte et al. Furthermore, we showcase several illustrative examples across various dimensions $n$ and select values of $N$ that yield optimal configurations. The findings illuminate the intricate structure of optimal point distributions on spheres, thereby enriching the existing body of research in this domain.
\end{abstract}

\section{Introduction}

\subsection{The Tammes Problem}
The Tammes problem  is an important problem in the field of geometry and optimization that deals with the arrangement of points on the surface of a sphere. It was first asked by the Dutch botanist Tammes  \cite{Tammes}, while examining the distribution of openings on the pollen grains of different
flower species. The problem can be formally described as follows: Given a sphere in $\mathbb{R}^3$, the goal is to determine the optimal arrangement of $N$ points on the surface of the sphere such that the minimum distance between any two points is maximized. This is equivalent to maximizing the smallest angle between the lines connecting the center of the sphere to each of the points on its surface. Interested readers can refer
to \cite{Brass, Musion_overview} for a comprehensive overview. 

The classical Tammes problem focuses on $3$-dimensional space, and mathematically proven solutions are known for several specific values of $N$ (with $N \geq 3$, as the cases $N = 1$ and $N = 2$ are trivial). Some examples of these solutions are presented in Table \ref{Table 1} (for more examples, please refer to \cite{cohn_table}):

\begin{table}[h]
\centering
\caption{Tammes problem for specific values of $N$ in  $3$-dimensional case}
\label{Table 1}
\begin{tabular}{|c|c|}
\hline 
N & Historical works\tabularnewline
\hline 
\hline 
$3,4,6,12$ & L. Fejes Toth {[}19{]}\tabularnewline
\hline 
$5,7,8,9$ & Sch{\"u}tte and van der Waerden \cite{van de Waerden 1951}\tabularnewline
\hline 
$10,11$ & Danzer \cite{Danzer}, B{\"o}r{\"o}czky
\cite{Borozky}, L. H{\'a}rs \cite{H=00007B=00005C'a=00007D=000301rs}\tabularnewline
\hline 
$13,14$ & Musin and Tarasov \cite{Musin_13,musin_14}\tabularnewline
\hline 
24 & Robinson \cite{robinson}\tabularnewline
\hline 
\end{tabular}
\end{table}

This problem can be elegantly extended to encompass the $n$-dimensional case, allowing for a richer exploration of geometric configurations in higher-dimensional spaces. 
The goal is to maximize the minimum pairwise distance  between the $N$ points placed on the surface of the $n$-dimensional sphere:
\begin{equation}\label{tammes}
d_{n,N} = \max\left\{ \min_{1 \leq i < j \leq N} \| \boldsymbol{x}_i - \boldsymbol{x}_j \|_2\ :\ \boldsymbol{x}_i\in \mathbb{S}^{n-1},\  i=1,\ldots,N\right\},
\end{equation}
where $\mathbb{S}^{n-1}$ is the $n$-dimensional unit sphere. 
When the dimension parameter $n$ is fixed, $d_{n,N}$ can be succinctly simplified to $d_{N}$. In other words, the problem can be articulated as follows: how to optimally distribute $N$ points on the unit sphere $\mathbb{S}^{n-1}$ in such a manner that the minimum distance between any two distinct points  is maximized. The Tammes problem illustrates a fundamental aspect of spatial distribution and optimization, bridging geometry, physics, and computation. The configurations obtained from solving this problem provide insights into the natural arrangement of entities in $n$-dimensional space.


\subsection{The Kissing Number and Spherical Codes}

A closely associated problem is the kissing number problem.  Denote the kissing number in $n$-dimensional case as  $k(\mathbb{B}^n)$, where $\mathbb{B}^n$ denote the unit ball in $\mathbb{R}^n$. Formally, the kissing number problem can be mathematically represented as:
\begin{equation}\label{kissing}
k(\mathbb{B}^n) : = \max \{ N\in \mathbb{N}^{+}\ : \  \| \boldsymbol{x}_i - \boldsymbol{x}_j \|_2 \geq 2, \quad \text{for all } 1\leq i<j\leq N,\ \text{and}\ \boldsymbol{x}_i\in \mathbb{R}^n\ \text{with}\ \|\boldsymbol{x}_i\|_2=2,\  i=1,\ldots,N \}.\end{equation} It denotes the maximum number of non-overlapping unit balls $\mathbb{B}^n$ that can simultaneously contact a central unit ball $\mathbb{B}^n$ at its boundary in $\mathbb{R}^n$. The points $\boldsymbol{x}_i$ in equation (\ref{kissing}) represent the centers of the unit balls $\mathbb{B}^n$ that surrounding the central unit ball. For more knowledge about kissing number we refer to \cite{Borozky, Brass, Conway and Sloane, odlyzko, Pfender and Zregler, zong2024, zong, zong_kissing}.

In the $3$-dimensional case, this problem became the focal point of a notable discourse between Isaac Newton and David Gregory in 1694. Historical accounts suggest that Newton postulated the answer to be 12 balls, while Gregory speculated that 13 might be feasible. This intriguing conundrum is often referred to as the "{\em thirteen spheres problem}" \cite{key-19, musin_3, zong}. Ultimately, this problem was resolved by Sch{\"u}tte  and van de Waerden in 1951 \cite{van de Waerden 1951}.

The kissing number problem can be viewed as a particular instance of the spherical code design problem. For a given fixed distance $d$, the objective of the spherical code design problem is to ascertain the maximum number of points $N$ that can be positioned on the surface of a unit sphere $\mathbb{S}^{n-1}$ such that the distance between any two distinct points is at least $d$. More formally, the goal is to maximize $N$, the number of points, subject to the specified distance constraint  \cite{zong}:
\begin{equation}\label{sphere code}
     m[n,d]:=\max\{  N\in \mathbb{N}^{+} \ : \  \| \boldsymbol{x}_i - \boldsymbol{x}_j \|_2 \geq d, \quad \text{for all } 1\leq i<j\leq N,\ \text{and}\ \boldsymbol{x}_i\in \mathbb{S}^{n-1},\  i=1,\ldots,N \}.
 \end{equation}
If $N$ unit spheres are in contact with a central unit sphere in $\mathbb{R}^{n}$, then the set of kissing points forms an arrangement on the surface of the central sphere such that the Euclidean distance between any two points is at least $1$. Consequently, we can directly conclude that 
\[k(\mathbb{B}^n) = m[n, 1].\]
This observation allows us to reformulate the kissing number problem in an alternative manner by asking how many points can be positioned on the surface of $\mathbb{S}^{n-1}$ such that the distance between any pair of points is no less than $1$.  For more knowledge about spherical codes we refer to \cite{Andreev, Conway and Sloane, Delsarte, key-26, zong}

\subsection{The Relationship Between the Tammes Problem and Sphere Code}

By comparing (\ref{tammes}) with (\ref{sphere code}), we observe that the Tammes problem is intricately linked to the sphere code design problem. Specifically, for any fixed $n \in \mathbb{N}^{+}$ and $d > 0$, it follows that
\[d_{n,m[n,d]} \geq d.\]

Consequently, the thirteen spheres problem can be reformulated as a query in the context of the Tammes problem: to ascertain whether $d_{3,13} < 1$.
More specifically, if $d_{3,13} \geq 1$, then Gregory's assertion is upheld, indicating that $13$ spheres can be simultaneously tangent to a single sphere. Conversely, if $d_{3,13} < 1$, then Newton's assertion prevails, suggesting that a single sphere can be tangent to at most $12$ other spheres concurrently. This problem was examined in \cite{van de Waerden 1951}. However, the optimal value of $d_{3,13}$ was not established by Musin and Tarasov until 2012 \cite{Musin_13}. This underscores that, even when the optimal number of points $N$ is achieved in the sphere code design problem for specific values of $n$ and $d$, obtaining the optimal value of $d_{n,N}$ in the context of the Tammes problem remains  more challenging.

\subsection{Organization of the Paper}
In this paper, we focus on establishing conditions to ascertain the value of $d_{n,N}$ for general $n$ and $N$ through the utilization of linear programming techniques, originally introduced by Delsarte, Goethals, and Seidel \cite[Theorem 4.3 and Theorem 5.10]{Delsarte}. In Section \ref{sec: linear}, we present the linear programming framework, where Lemma \ref{Thm1} plays a pivotal role in the main theorem. Section \ref{sec: main} delineates the sufficient conditions for determining $d_{n,N}$ as articulated in Theorem \ref{thm: main_theorem1}. Furthermore, we provide illustrative examples that yield optimal arrangements of point sets for the Tammes problem for specific choices of $n$ and $N$, employing auxiliary functions that satisfy the conditions outlined in Theorem \ref{thm: main_theorem1}. Further details can be found in Section \ref{sec: example}.

\section{The Linear Programming
Method }\label{sec: linear}
Let us revisit the mathematical formulation of the Tammes problem. Denote $C = \{\boldsymbol{x}_1, \boldsymbol{x}_2, \ldots, \boldsymbol{x}_N\} \subset \mathbb{S}^{n-1}$ as a specific set of $N$ points on the unit sphere $\mathbb{S}^{n-1}$. Define $d_{C}$ and $t_{C}$ as follows:
\begin{equation}\label{eq: dc_tc}
d_{C} := \min_{1 \leq i < j \leq N} \|\boldsymbol{x}_{i} - \boldsymbol{x}_{j}\|_{2}, \quad \text{and} \quad t_{C} := \max_{1 \leq i < j \leq N} \boldsymbol{x}_{i} \cdot \boldsymbol{x}_{j},
\end{equation}
where $\boldsymbol{x}_{i} \cdot \boldsymbol{x}_{j}$ represents the inner product of the vectors $\boldsymbol{x}_{i}$ and $\boldsymbol{x}_{j}$. 
We can directly elucidate the relationship between $d_C$ and $t_C$:  \begin{equation} t_{C} = 1 - \frac{1}{2} d_{C}^{2}\label{eq:t_c and d_c}, \end{equation}
as 
$\|\boldsymbol{x}_{i}-\boldsymbol{x}_{j}\|_{2}^{2} = (\boldsymbol{x}_{i}-\boldsymbol{x}_{j}) \cdot (\boldsymbol{x}_{i}-\boldsymbol{x}_{j}) = \boldsymbol{x}_{i} \cdot \boldsymbol{x}_{i} - 2 \boldsymbol{x}_{i} \cdot \boldsymbol{x}_{j} + \boldsymbol{x}_{j} \cdot \boldsymbol{x}_{j} = 2 - 2 \boldsymbol{x}_{i} \cdot \boldsymbol{x}_{j},$
for any $\boldsymbol{x}_{i}, \boldsymbol{x}_{j} \in \mathbb{S}^{n-1}$. 
Thus the objective of the Tammes problem in the $n$-dimensional case is to maximize $d_C$ across all configurations of $C$ consisting of $N$ distinct points on the sphere $\mathbb{S}^{n-1}$. Consequently, the maximum distance $d_{n,N}$, as designated in (\ref{tammes}), can be expressed as follows: \begin{equation} d_{n,N} = \max_{C \in \mathcal{C}} d_C=\max_{C\in \mathcal{C}}\sqrt{2-2t_{C}}, \label{eq:d_n,N}\end{equation} where $\mathcal{C}$ represents the collection of all point sets $C$ comprising $N$ distinct points on the surface of the sphere $\mathbb{S}^{n-1}$.

To date, for the case when $n=3$, optimal solutions to the classical Tammes problem have been rigorously established only for relatively small instances, specifically for values of $N$ ranging from $1$ to $14$ and $24$. These solutions were primarily derived through constructive methods grounded in spherical geometry. For larger instances, many researchers have shifted their focus toward the pursuit of high-quality suboptimal solutions rather than establishing the optimality of the solutions. In addition to the constructive techniques, various numerical global optimization methods have been proposed \cite{Lai, Mackay}. The methodologies employed encompass both construction approaches informed by prior knowledge of the problem and numerical global optimization strategies.

In this paper, we concentrate on deriving the optimal solution for the Tammes problem through the application of linear programming techniques. The linear programming approach for deriving bounds on spherical codes was developed in parallel with the analogous methodology for codes over finite fields, pioneered by Delsarte \cite{key-13}. In 1977, Delsarte, Goethals, and Seidel \cite{Delsarte} established the principal theorem, which was subsequently generalized by Kabatianskii and Levenshtein \cite{Kabatyanskii} in 1978.
The Gegenbauer polynomials are instrumental in the linear programming method. Numerous definitions for these polynomials have been proposed \cite{Conway and Sloane,Delsarte,Pfender and Zregler}. For a given dimension $n$, they can be characterized by the following recurrence relations: $P_{0}^{(n)}(t) = 1$, $P_{1}^{(n)}(t) = t$, and
\[
P_{k+1}^{(n)}(t)=\frac{(2k+n-2)tP_{k}^{(n)}(t)-kP_{k-1}^{(n)}(t)}{k+n-2},\qquad\text{for}\ k\geq1,
\]
As observed, $P_{k}^{(n)}$ represents a special case of the Jacobi polynomial $P_{k}^{(\alpha,\beta)}$ with parameters $\alpha = \beta = \frac{n-3}{2}$.

We hereby define the set $P(k, \tau, n)$ in terms of the Gegenbauer polynomials, where $\tau \in [-1, 1)$.
\begin{defn}
\label{def2}Let $P(k, \tau, n)$ denote the set defined as follows: for any fixed $k$, $n$, and $\tau \in [-1, 1)$, a function $f$ belongs to $P(k, \tau, n)$ if and only if it satisfies the following two properties:

(i) $f(t) = \sum_{i=0}^{k} c_{i} P_{i}^{(n)}(t)$, where $P_{i}^{(n)}$ (for $i = 0, \ldots, k$) denotes the $i$-th degree Gegenbauer polynomials in the context of $n$-dimensional space. The coefficients $\{c_{i}\}_{i=0}^{k}$ must satisfy $c_{0} > 0$ and $c_{i} \geq 0$ for $i = 1, \ldots, k$;

(ii) $f(t) \leq 0$ for $t \in [-1, \tau].$

Additionally, let $f^{\#}$ be defined as follows: 
\begin{equation} 
f^{\#} := \frac{f(1)}{c_{0}}.\label{eq:f=000023} 
\end{equation}\end{defn}

The following lemma plays a crucial role in the proof of the main theorem.
\begin{lem}
\label{Thm1}Consider a fixed point set $C = \{ \boldsymbol{x}_{1}, \ldots, \boldsymbol{x}_{N} \} \subset \mathbb{S}^{n-1}$, with $t_{C}$ defined in (\ref{eq: dc_tc}). If there exists a function $f$ such that $f \in P(K, t_{C}, n)$ for some degree parameter $K$, then it follows that
\[
N\leq f^{\#},
\]
where $f^{\#}$ is defined in (\ref{eq:f=000023}). Moreover, if $N=f^{\#}$,
we have 

\[
f(\boldsymbol{x}_{i}\bm{\cdot}\boldsymbol{x}_{j})=0,
\]
for all $1\leq i<j\leq N$.
\end{lem}

\begin{proof}
Define
\[
S:=\sum_{i,j=1}^{N}f(\boldsymbol{x}_{i}\bm{\cdot}\boldsymbol{x}_{j}).
\]
Given that  $f\in P(K,t_{C},n)$, as per Definition \ref{def2}, it follows that 
\begin{equation}
S=\sum_{i,j=1}^{N}\sum_{l=0}^{K}c_{l}P_{l}^{(n)}(\boldsymbol{x}_{i}\bm{\cdot}\boldsymbol{x}_{j})=N^{2}c_{0}+\sum_{i,j=1}^{N}\sum_{l=1}^{K}c_{l}P_{l}^{(n)}(\boldsymbol{x}_{i}\bm{\cdot}\boldsymbol{x}_{j}).\label{eq: S_equ}
\end{equation}
Introduce
\[
Q=\sum_{i,j=1}^{N}\sum_{l=1}^{K}c_{l}P_{l}^{(n)}(\boldsymbol{x}_{i}\bm{\cdot}\boldsymbol{x}_{j})=\sum_{l=1}^{K}c_{l}\sum_{i,j=1}^{N}P_{l}^{(n)}(\boldsymbol{x}_{i}\bm{\cdot}\boldsymbol{x}_{j}).
\]
Utilizing the positive definiteness property of Gegenbauer polynomials (as established in \cite[Lemma 1]{musin_3} and \cite[Page 126]{zong}), along with the non-negativity condition $c_{l} \geq 0$ for $l = 1, \ldots, K$, it can be concluded that
\begin{equation}
Q\geq0.\label{(1)}
\end{equation}
Substituting (\ref{(1)}) into (\ref{eq: S_equ}) yields the inequality
\begin{equation}
S\geq N^{2}c_{0}.\label{(2)}
\end{equation}

Conversely, $S$ can be expressed as
\begin{equation}
S=\sum_{i,j=1}^{N}f(\boldsymbol{x}_{i}\bm{\cdot}\boldsymbol{x}_{j})=\sum_{i=1}^{N}f(\boldsymbol{x}_{i}\bm{\cdot}\boldsymbol{x}_{i})+\sum_{\underset{i\neq j}{i,j=1}}^{N}f(\boldsymbol{x}_{i}\bm{\cdot}\boldsymbol{x}_{j})=N\cdot f(1)+\sum_{\underset{i\neq j}{i,j=1}}^{N}f(\boldsymbol{x}_{i}\bm{\cdot}\boldsymbol{x}_{j}).\label{eq:another_S}
\end{equation}
Given that $f(t) \leq 0$ for $t \in [-1, t_{C}]$ and considering that $-1 \leq \boldsymbol{x}_{i} \cdot \boldsymbol{x}_{j} \leq t_{C}$ for $1 \leq i < j \leq N$, it follows that
\begin{equation}
\sum_{\underset{i\neq j}{i,j=1}}^{N}f(\boldsymbol{x}_{i}\bm{\cdot}\boldsymbol{x}_{j})\leq0.\label{(3)}
\end{equation}
This leads to the conclusion that 
\begin{equation}
S\leq N\cdot f(1).\label{(4)}
\end{equation}
By combining inequalities (\ref{(2)}) and (\ref{(4)}), one can deduce that
\[
N\leq\frac{f(1)}{c_{0}}=f^{\#}.
\]

Moreover, if $N=f^{\#}=\frac{f(1)}{c_{0}}$, we have $N^{2}c_{0}=Nf(1)$.
Based on the representation of $S$ in (\ref{eq: S_equ}) and (\ref{eq:another_S}),
it can be immediately inferred that
\[
\sum_{i,j=1}^{N}\sum_{l=1}^{k}c_{l}P_{l}^{(n)}(\boldsymbol{x}_{i}\bm{\cdot}\boldsymbol{x}_{j})=\sum_{\underset{i\neq j}{i,j=1}}^{N}f(\boldsymbol{x}_{i}\bm{\cdot}\boldsymbol{x}_{j}).
\]
Then according to (\ref{(1)}) and (\ref{(3)}), it follows that 
\[
\sum_{\underset{i\neq j}{i,j=1}}^{N}f(\boldsymbol{x}_{i}\bm{\cdot}\boldsymbol{x}_{j})=0.
\]
Since $f(t)\leq0$ for $t\in[-1,t_{C}]$, we can assert that
\[
f(\boldsymbol{x}_{i}\bm{\cdot}\boldsymbol{x}_{j})=0,
\]
for all indices where $1 \leq i < j \leq N$.
\end{proof}


\section{Main Theorem}\label{sec: main}

In this section, we elucidate the principal findings derived from the application of linear programming techniques.

\begin{thm}
\label{thm: main_theorem1}Let $C = \{\boldsymbol{x}_{1}, \ldots, \boldsymbol{x}_{N}\} \subset \mathbb{S}^{n-1}$ be a fixed point set, with $t_{C}$ and $d_{C}$ defined as in (\ref{eq: dc_tc}).
If for the set $C$ satisfies the following three conditions:

(i) There exists a function $f \in P(k_1, t_{C}, n)$ such that $N = f^{\#}$;


(ii) There exists $t_{2} \in [-1, t_{C})$ such that $f(t) \neq 0$ for $t \in (t_{2}, t_{C})$;

(iii) There exists a function $g \in P(K_{2}, t_{2}, n)$ such that $N > g^{\#}$.

then the set $C$ constitutes a solution to the Tammes problem for the parameters $n$ and $N$, and we have
\[
d_{n,N}=d_{C}.
\]
\end{thm}

\begin{proof}[\textbf{Proof of Theorem \ref{thm: main_theorem1}}]

Assume the existence of another point set $S = \{ \widetilde{\boldsymbol{x}}_{1}, \ldots, \widetilde{\boldsymbol{x}}_{N} \} \subset \mathbb{S}^{n-1}$, with $d_{S}$ and $t_{S}$ defined as in (\ref{eq: dc_tc}). If it is the case that $d_{S} > d_{C}$, then, consistent with (\ref{eq:t_c and d_c}), we ascertain
\begin{equation}
t_{S} < t_{C}. \label{eq:1(1)_temp}
\end{equation}
To analyze this, we will delineate our examination into two distinct scenarios based on conditions (ii) and (iii): $t_{S} \leq t_{2}$ and $t_{S} > t_{2}$.

\textbf{Case 1: $t_{S} \leq t_{2}$.}
In this scenario, we can directly infer that 
\[g \in P(K_{2}, t_{2}, n) \subseteq P(K_{2}, t_{S}, n).\] By invoking Lemma \ref{Thm1}, we find $N \leq g^{\#}$, which contradicts condition (iii).

\textbf{Case 2: $t_{S} > t_{2}$.}
 Based on condition
(ii), taking $\widetilde{\boldsymbol{x}}_{k_{1}}$ and $\widetilde{\boldsymbol{x}}_{k_{2}}$
such that $\widetilde{\boldsymbol{x}}_{k_{1}}\bm{\cdot}\widetilde{\boldsymbol{x}}_{k_{2}}=t_{S}$,
we have 
\begin{equation}
f(\widetilde{\boldsymbol{x}}_{k_{1}}\bm{\cdot}\widetilde{\boldsymbol{x}}_{k_{2}})\neq 0.\label{eq:f_k}
\end{equation}
However, considering condition (i) and the fact that $t_{S} < t_{C}$, it follows that \[f \in P(k_1, t_{C}, n) \subseteq P(k_1, t_{S}, n).\] Furthermore, given that $N = f^{\#}$ as stated in condition (i) and applying Lemma \ref{Thm1}, we can deduce
that \[f(\widetilde{\boldsymbol{x}}_{k_{1}}\bm{\cdot}\widetilde{\boldsymbol{x}}_{k_{2}})=0.\]
It contradicts with (\ref{eq:f_k}). 

Therefore, for any point set $S=\{\widetilde{\boldsymbol{x}}_{1},\cdots,\widetilde{\boldsymbol{x}}_{N}\}\subset\mathbb{S}^{n-1}$,
we have $d_{S}\leq d_{C}$ and therefore
\[
d_{n,N}=d_{C}.
\]
\end{proof}

\begin{rem}

Indeed, condition (ii) and (iii) indicate that $f$ has multiple zero points in addition to $t = t_{C}$, with $t_{2}$ representing the largest zero point in $[-1, t_C)$.
\end{rem}

\section{Examples and Discussions}\label{sec: example}

In this section, we address the Tammes problem for specific selections of $n$ and $N$, deriving the exact values of $d_{n,N}$. To ensure the completeness of this paper, we present the proofs of these results utilizing Theorem \ref{thm: main_theorem1}.

\begin{example}
The distance $d_{n,2n}$ is given by $d_{n,2n} = \sqrt{2}$, and a point set solution is represented by the $n$-dimensional cross polytope
$C = \{\pm \boldsymbol{e}_{1}, \ldots, \pm \boldsymbol{e}_{n}\},$
where $\boldsymbol{e}_{1}, \ldots, \boldsymbol{e}_{n}$ form an orthonormal basis of $\mathbb{R}^{n}$.
\end{example}

\begin{proof}
When $n = 1$, the result can be established trivially. Henceforth, attention is restricted to the case where $n \geq 2$.
According to the definition of the set $C$, it follows that $|C| = N = 2n$, $d_{C} = \sqrt{2}$, and $t_{C} = 0$. Consider the function
\[f(t) = t(t + 1).\]
Through direct computation, it can be shown that $f \in P(2, 0, n)$ and $f(t) \leq 0$ for $t \in [-1, t_{C}]$. Moreover, it holds that
\[f^{\#} = \frac{f(1)}{c_{0}} = \frac{2}{\frac{1}{n}} = 2n = N.\]
Defining $t_{2} = -1$ leads to
\[g(t) = t + 1.\]
It is evident that $f(t) \neq 0$ for $t \in (t_{2}, t_{C})$, and $g(t) \in P(1, -1, n)$ such that
\[N = 2n > g^{\#} = 2,\]
given that $n \geq 2$. Consequently, the function $f$ satisfies conditions (i) and (ii), while $g$ meets condition (iii) as outlined in Theorem \ref{thm: main_theorem1}. Therefore, it can be concluded that
\[d_{n,2n} = d_{C} = \sqrt{2 - 2t_{C}} = \sqrt{2}.\]\end{proof}
\begin{example}
The distance $d_{3,12}$ is given by
$d_{3,12} = \frac{4}{\sqrt{10 + 2\sqrt{5}}}, $
and a point set solution is represented by
$C = \{\text{the vertices of a regular icosahedron}\} \subset \mathbb{S}^{2}.$
\end{example}

\begin{proof}
According to the definition of the set $C$, it follows that $|C| = N = 12$. Furthermore, three distinct values are found in the set
\[\{\|\boldsymbol{x}_{i} - \boldsymbol{x}_{j}\|_{2} : \boldsymbol{x}_{i}, \boldsymbol{x}_{j} \in C \text{ and } \boldsymbol{x}_{i} \neq \boldsymbol{x}_{j}\}.\]
These distances are expressed as
\[d_{C} = d_{1} = \frac{4}{\sqrt{10 + 2\sqrt{5}}}, \quad d_{2} = \frac{1}{5}\sqrt{10(5 + \sqrt{5})}, \quad \text{and} \quad d_{3} = 2.\]
Define
\[t_{C} = t_{1} = \frac{\sqrt{5}}{5}, \quad t_{2} = -\frac{\sqrt{5}}{5}, \quad \text{and} \quad t_{3} = -1.\]
The function $f$ can be constructed as
\begin{align*}
  f(t) &= (t + 1)\left(t + \frac{\sqrt{5}}{5}\right)^{2}\left(t - \frac{\sqrt{5}}{5}\right)\\
  &= t^{4} + \frac{\sqrt{5} + 5}{5}t^{3} + \frac{\sqrt{5} - 1}{5}t^{2} + \frac{-\sqrt{5} - 5}{5}t - \frac{\sqrt{5}}{25}\\
  &=\frac{8}{35}P_4^{(3)}+\frac{10+2\sqrt{5}}{25}P_3^{(3)}+\frac{46+14\sqrt{5}}{105}P_2^{(3)}+\frac{40+8\sqrt{5}}{25}P_1^{(3)}+\frac{10+2\sqrt{5}}{75}P_0^{(3)}.  
\end{align*}
Through direct calculation, it can be verified that $f \in P(4, \frac{\sqrt{5}}{5}, 3)$ and $f(t) \leq 0$ for $t \in [-1, t_{C}]$. Additionally,
\[f^{\#} = \frac{f(1)}{c_{0}} = \frac{2\left(1 + \frac{\sqrt{5}}{5}\right)^{2}\left(1 - \frac{\sqrt{5}}{5}\right)}{\frac{2\sqrt{5} + 10}{75}} = 12 = N.\]
Moreover, it holds that $f(t) < 0$ for $t \in (t_{2}, t_{C})$.
Next, consider the function
\begin{align*}
 g(t) &= (t + 1)\left(t + \frac{\sqrt{5}}{5}\right)\\
 &= t^{2} + \frac{5 + \sqrt{5}}{5}t + \frac{\sqrt{5}}{5}\\
 &=\frac{2}{3}P_2^{(3)}+\frac{5+\sqrt{5}}{5}P_1^{(3)}+\frac{5+3\sqrt{5}}{15}P_0^{(3)}.   
\end{align*}
It follows that $g(t) \in P(2, -\frac{\sqrt{5}}{5}, 3)$ and
\[N = 12 > g^{\#} = \frac{12}{5}\sqrt{5}.\]
Consequently, the function $f$ satisfies conditions (i) and (ii), while $g$ fulfills condition (iii) as stated in Theorem \ref{thm: main_theorem1}. Therefore, it can be concluded that
\[d_{3,12} = d_{C} = \sqrt{2 - 2t_{C}} = \frac{4}{\sqrt{10 + 2\sqrt{5}}}.\]
\end{proof}

\begin{example}
The distance $d_{4,120}$ is given by
$d_{4,120} = \frac{\sqrt{5} - 1}{2},$
and a point set solution is represented by
$C = \{\text{the vertices of a 600-cell}\} \subset \mathbb{S}^{3}.$
\end{example}

\begin{proof}
The $600$-cell is a finite regular $4$-dimensional polytope with
120 vertices and 720 edges. Therefore, $|C|=N=120$. Additionally, there exist 8 distinct values in the setz
\[\{\|\boldsymbol{x}_{i} - \boldsymbol{x}_{j}\|_{2} : \boldsymbol{x}_{i}, \boldsymbol{x}_{j} \in C \text{ and } \boldsymbol{x}_{i} \neq \boldsymbol{x}_{j}\}.\]
These distances are given by
\[d_{C} = d_{1} = \frac{\sqrt{5} - 1}{2}, \ d_{2} = 1, \quad d_{3} = \frac{\sqrt{2(-\sqrt{5} + 5)}}{2}, \ d_{4} = \sqrt{2}, \  d_{5} = \frac{\sqrt{5} + 1}{2}, \  d_{6} = \sqrt{3}, \  d_{7} = \frac{\sqrt{2(\sqrt{5} + 5)}}{2}, \  d_{8} = 2.\]
Define the parameters as
\[t_{C} = t_{1} = \frac{\sqrt{5} + 1}{4}, \quad t_{2} = \frac{1}{2}, \quad t_{3} = \frac{\sqrt{5} - 1}{4}, \quad t_{4} = 0, \quad t_{5} = -\frac{\sqrt{5} - 1}{4}, \quad t_{6} = -\frac{1}{2}, \quad t_{7} = -\frac{\sqrt{5} + 1}{4}, \quad t_{8} = -1.\]
The function $f$ can be constructed as
\begin{align*}
    f(t) &= 330825728(t + 1)^{2}t^{2}(t^{2} - \frac{1}{4})^{2}(t^{2} - \frac{3 - \sqrt{5}}{8})^{2}(t + \frac{1 + \sqrt{5}}{4})^{2}(t - \frac{1 + \sqrt{5}}{4})(t^{2} - \frac{9023 + 682\sqrt{5}}{5048}t + \frac{15649 + 3121\sqrt{5}}{20192})\\
    &= 45432P_{17}^{(4)}+(39695+9860\sqrt{5})P_{16}^{(4)}+(74208+4640\sqrt{5})P_{15}^{(4)} + (67485-870\sqrt{5})P_{14}^{(4)}+(-54120+38280\sqrt{5})P_{11}^{(4)}\\
    &+(34342+103356\sqrt{5})P_{10}^{(4)}+(155020+211700\sqrt{5})P_{9}^{(4)}+(355338+335124\sqrt{5})P_{8}^{(4)}+(561568+445440\sqrt{5})P_7^{(4)}\\
    &+(693868+515214\sqrt{5})P_6^{(4)}+(735888+518520\sqrt{5})P_5^{(4)}+(650135+457330\sqrt{5})P_4^{(4)}+(509144+346840\sqrt{5})P_3^{(4)}\\
    &+(304377+222546\sqrt{5})P_2^{(4)}+(154460+104980\sqrt{5})P_1^{(4)}+(36360+27840\sqrt{5})P_0^{(4)}.
\end{align*}
Through direct computation, it can be confirmed that $f \in P(17, \frac{\sqrt{5} + 1}{4}, 4)$ and $f(t) \leq 0$ for $t \in [-1, t_{C}]$. Furthermore,
\[f^{\#} = \frac{f(1)}{c_{0}} = \frac{f(1)}{36360+27840\sqrt{5}} = 120 = N.\]
Additionally, it holds that $f(t) < 0$ for $t \in (t_{2}, t_{C})$.
Next, consider the function
\begin{align*}
g(t) &= 330825728(t + 1)^{2}t^{2}(t + \frac{1}{2})^{2}(t^{2} - \frac{3 - \sqrt{5}}{8})^{2}(t + \frac{1 + \sqrt{5}}{4})^{2}(t - \frac{1}{2})\\
    &=565376P_{13}^{(4)}+(3149952 + 524992\sqrt{5})P_{12}^{(4)}+(10903680 + 3149952\sqrt{5})P_{11}^{(4)}+(29540896 + 10439264\sqrt{5})P_{10}^{(4)}\\
    &+(65422080 + 25441920\sqrt{5})P_{9}^{(4)}
    +(121939488 + 49975200\sqrt{5})P_{8}^{(4)}+(195135488 + 82383360\sqrt{5})P_7^{(4)}\\
    &+(270956448 + 116326112\sqrt{5})P_6^{(4)}
    +(327231552+141868992\sqrt{5})P_5^{(4)}+(341648640 + 149016960\sqrt{5})P_4^{(4)}\\
    &+(302799232 + 132540288\sqrt{5})P_3^{(4)}
    +(218376480 + 95770656\sqrt{5})P_2^{(4)}+(115619392 + 50762688\sqrt{5})P_1^{(4)}\\
    &+(32064896 + 14094016\sqrt{5})P_0^{(4)}.
\end{align*}

It follows that $g(t) \in P(13, \frac{1}{2}, 4)$ and
\[N = 120 > g^{\#} = \frac{g(1)}{32064896 + 14094016\sqrt{5}} = \frac{7200}{21431}(323 - 61\sqrt{5}).\]
Thus, the function $f$ satisfies conditions (i) and (ii), while $g$ meets condition (iii) as outlined in Theorem \ref{thm: main_theorem1}. Consequently, it can be concluded that
\[d_{4,120} = d_{C} = \sqrt{2 - 2t_{C}} = \frac{\sqrt{5} - 1}{2}.\]\end{proof}

\begin{rem}
The computations of the Gegenbauer polynomials expansion 
can be made with the help of the computer software product Maple.
\end{rem}

\end{document}